\documentclass[reqno, 12pt]{amsart}

\newenvironment{nouppercase}{
  
  \renewcommand{\uppercasenonmath}[1]{}}{}
  
\usepackage{amsmath,amssymb,amsthm,amsfonts,graphicx, array}
\usepackage{fullpage}
\usepackage{xcolor}
\usepackage[colorlinks=true,linkcolor=blue,anchorcolor=blue,citecolor=red]{hyperref}

\allowdisplaybreaks 

\newtheorem{theorem}{Theorem}[section]
\newtheorem{lemma}[theorem]{Lemma}

\newtheorem{proposition}[theorem]{Proposition}

\newcommand{\Ptwo}{\mathcal{P}_2}
\newcommand{\W}{\mathfrak{W}}
\newcommand{\N}{\mathbb{N}}
\newcommand{\Nzero}{\mathbb{N}_0}
\newcommand{\one}{\mathbf{1}}
\newcommand{\eps}{\varepsilon}

\author{Yuchen Ding}
\address{School of Mathematics, Yangzhou University, Yangzhou 225002, People's Republic of China;
HUN-REN Alfr\'ed R\'enyi Institute of Mathematics, Budapest, Pf.~127, H-1364 Hungary}
\email{Yuchen Ding: ycding@yzu.edu.cn}

\author{Huixi Li}
\address{School of Mathematical Sciences and LPMC, Nankai University, Tianjin 300071, China}
\email{Huixi Li: lihuixi@nankai.edu.cn}

\author{Junfeng Li}
\address{School of Mathematical Sciences and LPMC, Nankai University, Tianjin 300071, China}
\email{Junfeng Li: 
junfengli.math@gmail.com}

\makeatletter
\@namedef{subjclassname@2020}{\textup{}2020 Mathematics Subject Classification}
\makeatother

\title{A Romanoff-type theorem for \texorpdfstring{$\mathcal P_2+\{a^a:a\ge 1\}$}{P2
plus powers a to the a}}

\begin{document}

\begin{abstract}
Let $\Omega(n)$ denote the number of prime factors of $n$, counted with multiplicity, and put $\Ptwo=\{m\ge 1:\Omega(m)\le 2\}$. 
We prove that the sumset $\Ptwo+\{a^a:a\ge 1\}$ has positive lower density. The proof uses the Romanoff second moment method, in the spirit of Li and Pan's theorem on $\Ptwo+2^{\mathcal P}$. The main new ingredient is the following average estimate for the singular factor
\[
  \frac{1}{K(K-1)}
  \sum_{\substack{1\le a,b\le K\\a\ne b}}
  \prod_{p\mid a^a-b^b}\left(1+\frac{\kappa}{p}\right)
  \le C_\kappa
\]
for some constant $C_\kappa>0$, which is valid for all $K \ge 2$ and any fixed $\kappa>0$. This estimate controls the average arithmetic correlation among the shifts $a^a$ and allows the Romanoff argument to be carried out.
\end{abstract}

\begin{nouppercase}
\maketitle
\end{nouppercase}

\section{Introduction}
In 1849, de Polignac \cite{de1851} conjectured that every odd number except \(3\) can be represented as the sum of a prime and a power of \(2\). Although counterexamples were later found, the problem has remained a useful source of questions in additive number theory. Romanoff's theorem, proved in 1934, asserts that the set \(\{2^k+p:k\in\Nzero,\ p\in\mathcal P\}\) has positive lower density \cite{romanov1934}. In 2004, Chen and Sun \cite{chen2004romanoff} gave the explicit lower bound \(0.0868\). Further improvements of this lower bound can be found in \cite{Lv2007, HabsiegerRoblot2006, Pintz2006, HS2010, elsholtz2018romanov, johnston2026}.

For upper density, van der Corput \cite{van1950} proved in 1950 that the upper density is strictly less than \(0.5\). In the same year, Erd\H{o}s \cite{erdos1950} used covering systems to show that there is an arithmetic progression with common difference \(11184810\) whose terms cannot be expressed as the sum of a prime and a power of \(2\). Habsieger and Roblot \cite{HabsiegerRoblot2006} improved the upper density bound to \(0.4909\), and Chen, Dai, and Li \cite{chen2025some} recently improved it further to \(0.490341\). Let \(U\) denote the set of positive integers which are not of the form \(p+2^k\). The structure of \(U\) is also of independent interest. 
We refer to Erd\H{o}s' 1995 conjecture on \(U\) \cite{Erdos1995} and to the recent works \cite{chen2025some, ChenErdosConjecture, ChenMinimalDifference}.
For the asymptotic density, Romani \cite{romani1983computations} conjectured in 1983 that the density of odd integers representable as $p+2^k$ is approximately $0.434$. In 2020, del Corso et al. \cite{del2020computing} refined this conjectural value to approximately $0.437$.

Further lower density results for sumsets involving primes and sparse sequences can be found in \cite{Sun2004midu, Lee2010, BL2013, Dubickas2013, CW2024, LX2021, WC2023, chen2025arithmetic}. Romanoff-type problems have also been studied over polynomial rings; see \cite{shparlinski2017explicit, ding2022extending, ding2022problem, ding2026generalization, DP2026}. 

Let $\Ptwo=\{m\ge 1:\Omega(m)\le 2\}$ be the set of two almost primes, 
where $\Omega(m)$ denotes the number of prime factors of $m$, counted with
multiplicity, and where we adopt the convention $\Omega(1)=0$. 
In 2008, Li and Pan \cite{li2008romanoff} proved that $\Ptwo+2^{\mathcal P}
  =\{q+2^p:q\in\Ptwo,\ p\in\mathcal P\}$ has positive lower density. 
The present paper proves a Romanoff-type positive density theorem for the sumset
$$
\Ptwo+\{a^a:a \ge 1\}
  =\{q+a^a:q\in\Ptwo,\ a \ge 1\}.
  $$
This is a natural comparison, since both
\(\{2^p:p\in\mathcal P,\ 2^p\le N\}\) and
\(\{a^a:a\ge1,\ a^a\le N\}\) have size
\(\asymp \log N/\log\log N\).
This order also occurs for other familiar
sparse sequences, such as \(g^{\mathcal P}\) with fixed \(g\ge2\), factorials,
and primorials. 

The comparison above concerns only the size of the sparse sequence. For
\(a^a\), the main difficulty is local: in the second moment argument one must
control the average contribution of primes dividing the differences
\(a^a-b^b\). To measure this contribution, for a fixed nonzero integer \(h\)
and a real number \(\kappa>0\), define
\begin{equation}\label{eq:Wdef}
  \W_\kappa(h)=\prod_{p\mid h}\left(1+\frac{\kappa}{p}\right).
\end{equation}
The essential point is to prove that the singular factors attached to the
nonlinear differences \(a^a-b^b\) are bounded on average: 
\[
  \frac{1}{K(K-1)}
  \sum_{\substack{1\le a,b\le K\\a\ne b}}
  \W_\kappa(a^a-b^b)\ll_\kappa 1.
\]
This estimate shows that, although individual differences \(a^a-b^b\) may have many small prime divisors, their average sieve weight remains bounded.

The proof uses the period \(p(p-1)\) of the map \(x\mapsto x^x\pmod p\), 
together with separate estimates for the medium and large prime divisors of
\(a^a-b^b\). Combining this estimate with an upper-bound sieve for pairs of
two almost primes gives the following result.
\begin{theorem}\label{thm:main}
There exists a constant $\delta>0$ such that, when $N$ is sufficiently large,
\[
  \#\{n\le N:n=m+a^a,\ m\in\Ptwo,\ a\in\N\}\ge \delta N.
\]
\end{theorem}
The paper is organized as follows. 
Section~2 records the sieve estimates needed for correlations of two almost primes. Section~3 proves the average singular factor estimate for the differences $a^a-b^b$. 
Section~4 completes the second moment proof of the theorem.

The empty product in \eqref{eq:Wdef} is understood to be $1$. All implied constants are absolute unless a subscript indicates dependence on parameters.

\section{Correlations between two almost primes}
We collect the estimates needed for the first and second moment arguments.
Landau's theorem gives the size of \(\Ptwo\), while the next two lemmas give the upper bounds needed for correlations of two shifted \(\Ptwo\) conditions.

We begin with Landau's theorem for the number of integers with exactly \(r\) prime factors, counted with multiplicity.
\begin{lemma}[{\cite[Sec.~II.6.1]{tenenbaum2015introduction}}]\label{lem:landau}
As $X\to\infty$, for every fixed integer $r\ge 1$, we have 
\[
  \#\{n\le X:\Omega(n)=r\}
  \sim
  \frac{X}{\log X}\frac{(\log\log X)^{r-1}}{(r-1)!}, 
\]
and in particular
\begin{equation}\label{eq:P2-asymp}
  \#\{m\le X:m\in\Ptwo\}
  \sim \frac{X\log\log X}{\log X}.
\end{equation}
\end{lemma}

We shall also use the following elementary consequence of Mertens' estimates. Expanding the 
Euler product over primes \(p<z^\eta\), with \(\eta>0\) fixed and sufficiently
small, and using the weight \(\log d\) to discard the terms \(d\ge z\), one
obtains positive absolute constants \(B_1,B_2\) such that, for all \(z\ge3\)
and every nonzero integer \(D\), we have 
\[
  \sum_{\substack{d<z\\ d\mid P(z)\\ (d,D)=1}}
  \frac{\mu^2(d)2^{\omega(d)}}{d}
  \ge
  B_1(\log z)^2
  \prod_{\substack{p\mid D\\ p<z}}
  \left(1+\frac{B_2}{p}\right)^{-1},
\]
where \(P(z)\) denotes the product of the odd primes below \(z\).

The next estimate is a uniform upper bound sieve for two non-proportional
linear forms. We keep the dependence on the coefficients and on the determinant explicit.
\begin{lemma}\label{lemuniformupperbound}
There exist absolute constants $C,c>0$ with the following property.
Let $T\ge1$, and let $L_1(t)=\alpha_1 t+\beta_1$ and $L_2(t)=\alpha_2 t+\beta_2$ 
be two nonconstant integral non-proportional linear forms.  Put $\Delta=\alpha_1\beta_2-\alpha_2\beta_1\ne0$. 
Then, for every interval $I$ of integers of length $T$,
\begin{equation}\label{equniformupperbound}
\#\{t\in I:L_1(t),L_2(t)\text{ are prime}\}
  \le
  C\left(\frac{T}{\log^2(2+T)}+1\right)
  \prod_{p\mid \alpha_1\alpha_2\Delta}
  \left(1+\frac{c}{p}\right).    
\end{equation}
\end{lemma}

\begin{proof}
We may assume that \(\alpha_1\alpha_2\ne0\). 
If one of the two forms has a fixed prime divisor, say a prime \(r\) divides both coefficients of \(L_i\), then \(L_i(t)\) can be prime only when \(L_i(t)=r\). The number of such values of \(t\) is bounded absolutely, and is absorbed by the \(+1\) term. We shall therefore assume that neither form has a fixed prime divisor.

Let $F(t)=L_1(t)L_2(t)$. 
For an odd prime $p$, let $\rho(p)=\#\{t \pmod p:F(t)\equiv0\pmod p\}$. 
If $p\nmid \alpha_1\alpha_2\Delta$, then both $\alpha_1$ and $\alpha_2$
are invertible modulo $p$, and the two congruences $L_1(t)\equiv0\pmod p$, and $L_2(t)\equiv0\pmod p$ 
have the two distinct solutions $t\equiv -\beta_1\alpha_1^{-1}\pmod p$ and $t\equiv -\beta_2\alpha_2^{-1}\pmod p$, respectively. 
Hence $\rho(p)=2$ when $p \nmid 2\alpha_1\alpha_2\Delta$.  
For the remaining odd primes we only use the trivial bound $0\le \rho(p)\le 2$, which holds under the assumption that neither form has a fixed prime
divisor.  For squarefree $d$ composed of odd primes, define $\rho(d)$ multiplicatively. Then, for every interval $I$ of length $T$, we have 
\[
  \#\{t\in I:d\mid F(t)\}
  =
  \frac{\rho(d)}{d}T+ R_d
\]
with $|R_d| \le \rho(d)$. 

Let $N(I)=\#\{t\in I:L_1(t),L_2(t)\text{ are prime}\}$. 
Let $z=\max\{3,T^{1/4}\}$ and let $P(z)$ be the product of the odd primes below $z$. 
Let $S(z)=\#\{t\in I:(F(t),P(z))=1\}$. The values for which one of \(L_1(t)\) or \(L_2(t)\) is a prime not
exceeding \(z\) contribute at most \(A_1(z+1)\) values of \(t\), where
\(A_1>0\) is absolute.  Hence $N(I)\le S(z)+A_1(z+1)$. 

We now apply Selberg's upper bound sieve to the sequence $ \mathcal A=\{F(t):t\in I\}$ in the form of \cite[Theorem~3.1]{HR1974}.  There exist absolute
constants \(A_2,A_3>0\) such that
\[
  S(z)
  \le
  \frac{T}{G(z)}
  +
  A_2 z^2(\log z)^{A_3},
\]
where 
\[
G(z)
  =
  \sum_{\substack{d<z\\ d\mid P(z)}}
  \mu^2(d)
  \prod_{p\mid d}\frac{\rho(p)}{p-\rho(p)}.
\]
Since all terms in \(G(z)\) are non-negative, we may restrict the sum defining \(G(z)\) to squarefree \(d\) with \((d,2\alpha_1\alpha_2\Delta)=1\).  It follows for some absolute constants $A_4, C_1 > 0$ that 
\begin{align*}
G(z) 
&\ge
\sum_{\substack{d<z\\ d\mid P(z) \\ (d,2\alpha_1\alpha_2\Delta)=1}}
\mu^2(d)
\prod_{p\mid d}\frac{2}{p-2}. \\
&\ge
\sum_{\substack{d<z\\ d\mid P(z) \\ (d,2\alpha_1\alpha_2\Delta)=1}}
\mu^2(d)
\prod_{p\mid d}\frac{2}{p}. \\
&\ge
\sum_{\substack{d<z\\ d\mid P(z) \\ (d,2\alpha_1\alpha_2\Delta)=1}}
\frac{\mu^2(d) 2^{\omega(d)}}{d} \\
&\ge
A_4(\log z)^2
\prod_{\substack{p\mid 2 \alpha_1 \alpha_2 \Delta \\ p<z}}
\left(1+\frac{C_1}{p}\right)^{-1}.
\end{align*}
It follows that
\[
  \frac{T}{G(z)}
  \le
  \frac{1}{A_4}\,
  \frac{T}{(\log z)^2}
  \prod_{\substack{p\mid 2\alpha_1\alpha_2\Delta\\ p<z}}
  \left(1+\frac{C_1}{p}\right)
  \le
  \frac{1}{A_4}\,
  \frac{T}{(\log z)^2}
  \prod_{p\mid 2\alpha_1\alpha_2\Delta}
  \left(1+\frac{C_1}{p}\right).
\]
Recall that $z=\max(3,T^{1/4})$. There exists an absolute constant \(A_5>0\) such that
\[
  \frac{T}{(\log z)^2}
  \le
  A_5\left(\frac{T}{\log^2(2+T)}+1\right),
\]
and there exists an absolute constant \(A_6>0\) such that
\[
  A_2 z^2(\log z)^{A_3}+ A_1(z+1)
  \le
  A_6\left(\frac{T}{\log^2(2+T)}+1\right).
\]
Combining the preceding estimates, we obtain absolute constants
\(C,c>0\) such that 
\[
  \#\{t\in I:L_1(t),L_2(t)\text{ are prime}\}
  \le
  C
  \left(
    \frac{T}{\log^2(2+T)}+1
  \right)
  \prod_{p\mid \alpha_1\alpha_2\Delta}
  \left(1+\frac{c}{p}\right).
\]
This proves the lemma.
\end{proof}

We now pass from prime values of two linear forms to correlations of two shifted elements of \(\Ptwo\). The proof decomposes according to whether \(m\) and \(m+h\) are prime or products of two primes. The dependence on the shift \(h\) is collected in the factor \(\W_{\kappa_0}(h)\).
\begin{lemma}\label{lem:twodim}
There exist absolute constants \(C>0\) and \(\kappa_0>0\) such that, for every \(X\ge 3\) and every integer \(h\) with \(0<|h|\le X\), we have
\begin{equation}\label{eq:twodim}
\#\{m\le X:m+h\ge 1,\ m\in\Ptwo,\ m+h\in\Ptwo\}
\le
C
\frac{X(\log\log X)^2}{(\log X)^2}
\W_{\kappa_0}(h),
\end{equation}
where $\displaystyle\W_{\kappa_0}(h)
  =
  \prod_{p\mid h}\left(1+\frac{\kappa_0}{p}\right)$. 
\end{lemma}

\begin{proof}
Let \(c>0\) be the absolute constant in
Lemma~\ref{lemuniformupperbound}.  We shall choose an absolute constant
\(\kappa_1\ge c\), enlarged when necessary.  Put 
\[
  \mathcal D_X
  =
  \{p:p\le (2X)^{1/2},\ p\text{ prime}\}
\]
and let 
\[
  \mathcal N(X,h)
  =
  \#\{m\le X:m+h\ge1,\ m\in\Ptwo,\ m+h\in\Ptwo\}.
\]
It is clear that 
\[
  \#\{m\le X:m=1\text{ or }m+h=1\}\le 2. 
\]
For the two cases in which one of \(m,m+h\) is prime, we claim that
\begin{equation}\label{eq:prime-P2-side}
  \#\{m\le X:m\text{ prime},\ m+h\ge1,\ m+h\in\Ptwo\}
  \ll
  \W_{\kappa_1}(h)
  \frac{X\log\log X}{(\log X)^2}.
\end{equation}
By Lemma~\ref{lemuniformupperbound} we know
\[
  \#\{m\le X:m\text{ prime},\ m+h\text{ prime}\}
  \ll
  \W_{\kappa_1}(h)
  \frac{X}{(\log X)^2}. 
\]

If \(m+h\) is composite and belongs to \(\Ptwo\), then $m+h=eq$ with \(e\in\mathcal D_X\) and \(q\) prime. For each fixed \(e\), write \(q=t\). Then $m=et-h$, and the two primality conditions are imposed on the integral linear forms
\[
  L_{1,e}(t)=t,\qquad L_{2,e}(t)=et-h.
\]
Their determinant is $\Delta(L_{1,e},L_{2,e})
  =
  1\cdot(-h)-e\cdot0
  =
  -h$. 
The variable \(t\) ranges over an interval of length \(O(X/e+1)\).
Therefore Lemma~\ref{lemuniformupperbound} gives
\[
  \#\{m\le X:m\text{ prime},\ e\mid m+h,\ (m+h)/e\text{ prime}\}
  \ll
  \W_{\kappa_1}(h)
  \left(1+\frac{\kappa_1}{e}\right)
  \left(
    \frac{X/e}{\log^2(2+X/e)}+1
  \right).
\]
Summing over \(e\in\mathcal D_X\), and using $\log(2+X/e)\gg \log X$ for $e\le (2X)^{1/2}$, 
as well as $\sum_{e\in\mathcal D_X}\frac1e\ll\log\log X$ and $\#\mathcal D_X\ll \frac{X^{1/2}}{\log X}$, we obtain \eqref{eq:prime-P2-side}.

Similarly, we can prove 
\begin{equation}\label{eq:P2-prime-side}
  \#\{m\le X:m\in\Ptwo,\ m+h\text{ prime}\}
  \ll
  \W_{\kappa_1}(h)
  \frac{X\log\log X}{(\log X)^2}.
\end{equation}

We now treat the remaining terms, for which both \(m\) and \(m+h\) are
composite elements of \(\Ptwo\).  In this case we may write $m=du$ and $m+h=ev$, 
where \(d,e\in\mathcal D_X\) and \(u,v\) are prime.  Thus the remaining
contribution is at most $\sum_{d,e\in\mathcal D_X} T_{d,e}(X,h)$, where \(T_{d,e}(X,h)\) counts those \(m\) satisfying
\[
  m\le X,\qquad
  d\mid m,\qquad
  e\mid m+h,\qquad
  \frac md,\ \frac{m+h}{e}\text{ are both prime}.
\]

If \((d,e)\nmid h\), then \(T_{d,e}(X,h)=0\). When \((d,e)\mid h\), we put $g=(d,e)$ and $\ell=[d,e]$, and choose a residue class \(m_0\pmod\ell\) such that $m_0\equiv0\pmod d$ and $m_0\equiv -h\pmod e$. Every counted \(m\) is of the form $m=m_0+\ell t$, 
where \(t\) ranges over an interval of length \(O(X/\ell+1)\).  The two
primality conditions become
\[
  L_1(t)= \frac{\ell}{d} t + \frac{m_0}{d},
  \qquad
  L_2(t)= \frac{\ell}{e} t + \frac{m_0+h}{e}.
\]
These are integral nonconstant linear forms, and their determinant is $\Delta(L_1,L_2)
  =
  \frac{\ell}{d}\frac{m_0+h}{e}
  -
  \frac{\ell}{e}\frac{m_0}{d}
  =
  \frac{h}{g}
  \ne0$. 
Applying Lemma~\ref{lemuniformupperbound}, we get
\[
  T_{d,e}(X,h)
  \ll
  \left(
    \frac{X/\ell}{\log^2(2+X/\ell)}+1
  \right)
  \prod_{r\mid (\ell/d)(\ell/e)(h/g)}
  \left(1+\frac{c}{r}\right).
\]
The prime divisors of \((\ell/d)(\ell/e)\) are among the prime divisors
of \(de\), while the prime divisors of \(h/g\) are among those of \(h\).
Since \(d,e\) are primes, after enlarging \(\kappa_1\) we have
\[
  T_{d,e}(X,h)
  \ll
  \W_{\kappa_1}(h)
  \left(1+\frac{\kappa_1}{d}\right)
  \left(1+\frac{\kappa_1}{e}\right)
  \left(
    \frac{X/\ell}{\log^2(2+X/\ell)}+1
  \right).
\]
The factors involving \(d\) and \(e\) are bounded absolutely, because
\(d,e\ge2\).  Hence
\begin{equation}\label{eq:Tde-new}
  T_{d,e}(X,h)
  \ll
  \W_{\kappa_1}(h)
  \left(
    \frac{X/\ell}{\log^2(2+X/\ell)}+1
  \right).
\end{equation}

It remains to sum \eqref{eq:Tde-new}.  The contribution of the \(+1\)
term is
\[
  \ll
  \W_{\kappa_1}(h)\#\mathcal D_X^2
  \ll
  \W_{\kappa_1}(h)\frac{X}{(\log X)^2}.
\]

For the main term, first consider the diagonal case \(d=e=p\).  Then
\(\ell=p\), and
\[
  \W_{\kappa_1}(h)
  \sum_{p\le (2X)^{1/2}}
  \frac{X/p}{\log^2(2+X/p)}
  \ll
  \W_{\kappa_1}(h)
  \frac{X}{(\log X)^2}
  \sum_{p\le (2X)^{1/2}}\frac1p
  \ll
  \W_{\kappa_1}(h)
  \frac{X\log\log X}{(\log X)^2}.
\]
It remains to consider \(d=p\), \(e=q\), where \(p\ne q\) are primes and
\(p,q\le (2X)^{1/2}\).  Then \(\ell=pq\).  The part with
\(pq\le X^{1/2}\) is bounded by
\[
  \W_{\kappa_1}(h)
  \frac{X}{(\log X)^2}
  \sum_{p,q\le (2X)^{1/2}}\frac1{pq}
  \ll
  \W_{\kappa_1}(h)
  \frac{X(\log\log X)^2}{(\log X)^2}.
\]
For the complementary range \(pq>X^{1/2}\), partial summation together
with Chebyshev's estimate \(\pi(y)\ll y/\log y\) gives
\begin{align*}
  &\W_{\kappa_1}(h) \sum_{\substack{p,q\le (2X)^{1/2}\\ pq>X^{1/2}}}
  \frac{X/(pq)}{\log^2(2+X/(pq))}                                      \\
  &\ll \W_{\kappa_1}(h)
  X
  \int_{\log 2}^{(1/2)\log X+O(1)}
  \int_{\max(\log 2,(1/2)\log X-u+O(1))}^{(1/2)\log X+O(1)}
  \frac{du\,dv}
       {uv\{1+\max(0,\log X-u-v)\}^2}                                  \\
  &\ll \W_{\kappa_1}(h)
  \frac{X\log\log X}{(\log X)^2}.
\end{align*}
The last integral is elementary.  Write $u=\frac12\log X-r$, and $v=\frac12\log X-s$. In the central region the denominator contributes
\((1+r+s)^2\), giving $O\!\left(\frac{\log\log X}{(\log X)^2}\right)$, 
while the boundary regions where \(u\) or \(v\) is small give smaller
contributions.

Combining the estimates for \eqref{eq:prime-P2-side},
\eqref{eq:P2-prime-side}, the diagonal terms, and the off-diagonal terms, we obtain
\[
  \mathcal N(X,h)
  \ll
  \W_{\kappa_1}(h)
  \frac{X(\log\log X)^2}{(\log X)^2}.
\]
Finally, after increasing the absolute constants and putting
\(\kappa_0=\kappa_1\), this proves \eqref{eq:twodim}.
\end{proof}

\section{The average singular factor for \texorpdfstring{$a^a-b^b$}{a to the a minus b to the b}}
In this section we prove the estimate for the average singular factor which is used in the second moment argument.
\begin{proposition}\label{prop:avg-singular}
For every fixed $\kappa>0$ there exists a constant $C_\kappa>0$ such that, for all $K\ge 2$,
\begin{equation}\label{eq:avg-singular}
  \frac{1}{K(K-1)}
  \sum_{\substack{1\le a,b\le K\\a\ne b}}
  \prod_{p\mid a^a-b^b}\left(1+\frac{\kappa}{p}\right)
  \le C_\kappa.
\end{equation}
\end{proposition}
The proof separates the prime factors of \(a^a-b^b\) into three ranges. The following lemma supplies the estimate needed for the small primes.
\begin{lemma}\label{lem:one-prime-collision}
Let
\[
  B_p(K)=\#\{1\le a,b\le K:a^a\equiv b^b\pmod p\}.
\]
If $p\le K^{1/2}$, then
\begin{equation}\label{eq:Bp}
  B_p(K)\ll K^2\frac{\tau(p-1)}{p},
\end{equation}
where $\tau$ denotes the divisor function.
\end{lemma}

\begin{proof}
The case \(p=2\) is immediate, so we assume that \(p\) is odd. 
The map \(x\mapsto x^x\pmod p\) has period \(Q=p(p-1)\). 
Indeed, the base depends only on \(x\pmod p\). 
If \(p\nmid x\), the exponent may be reduced modulo \(p-1\), while if \(p\mid x\), then \(x^x\equiv0\pmod p\).

Since \(p\le K^{1/2}\), we have \(Q\le K\). 
Each residue class modulo \(Q\) occurs in \([1,K]\) at most $\lceil K/Q\rceil$ times. 
It follows that $B_p(K) \ll \left(\frac{K}{Q}\right)^2 M_p$, where
\[
  M_p=\#\{u,v \pmod Q: u^u\equiv v^v\pmod p\}.
\]
Here $u^u$ and $v^v$ are evaluated using any positive representatives of the residue classes modulo $Q$.

We now estimate $M_p$.  If $p\mid u$, then $u^u\equiv0\pmod p$.  There
are $p-1$ residue classes $u\pmod Q$ with $p\mid u$, and hence the
zero residue modulo $p$ contribution to $M_p$ is $(p-1)^2$. 

It remains to consider residue classes $u, v \pmod{Q}$ with $p\nmid u$ and $p\nmid v$. Choose a primitive root $g$ modulo $p$.  By the Chinese Remainder Theorem, a residue class $u\pmod {p(p-1)}$ with $p\nmid u$ is equivalent to a pair $s,e \pmod {p-1}$ with $u\equiv g^s \pmod p$ and $u\equiv e \pmod {p-1}$. 
For such a class we have $u^u\equiv (g^s)^e=g^{se}\pmod p$. 
Let
\[
  U_p
  =
  \#\{s,e,t,f \pmod {p-1}: se\equiv tf\pmod {p-1}\}.
\]
After $s,e,t$ are fixed, the congruence $tf\equiv se\pmod {p-1}$ has at most $(t,p-1)$ solutions in $f$.  Therefore
\[
  U_p
  \le
  (p-1)^2\sum_{t\pmod {p-1}}(t,p-1).
\]
Moreover, we have 
\begin{align*}
&\sum_{t\pmod {p-1}}(t,p-1) \\
&=
  \sum_{d \mid p-1} d\,
  \#\{t\pmod {p-1}:(t,p-1)=d\} \\
  &=
  \sum_{d\mid p-1} d\,\varphi\!\left(\frac{p-1}{d}\right) \\
  &=
  (p-1)\sum_{r\mid p-1}\frac{\varphi(r)}{r} \\
  &\le
  (p-1)\tau(p-1).
\end{align*}
Consequently, we have $U_p\le (p-1)^3\tau(p-1)$ and $M_p \ll
  p^3\tau(p-1)$. 
Substituting this into the preceding bound, and recalling that
$Q=p(p-1)$, we obtain $B_p(K) \ll K^2\frac{\tau(p-1)}{p}$. This proves the lemma.
\end{proof}

We now prove Proposition~\ref{prop:avg-singular}.
\begin{proof}
For $a\ne b$, set $h=a^a-b^b\ne 0$. When $1\le a,b\le K$, we have $|h|\le K^K$. 

Split the prime factors into three ranges $p\le K^{1/2}$, $K^{1/2}<p\le K$, and $p>K$. For the middle range, Mertens' estimate gives
\[
  \prod_{\substack{K^{1/2}<p\le K\\p\mid h}}
  \left(1+\frac{\kappa}{p}\right)
  \le
  \prod_{K^{1/2}<p\le K}
  \left(1+\frac{\kappa}{p}\right)
  \ll_\kappa 1.
\]
For the large primes, if $r$ distinct primes $p>K$ divide $h$, then $r < K$ since $K^r < K^K$. Hence
\[
  \sum_{\substack{p>K\\p\mid h}}\frac{1}{p}\le \frac{r}{K}\le 1,
\]
and therefore
\[
  \prod_{\substack{p>K\\p\mid h}}
  \left(1+\frac{\kappa}{p}\right)
  \le
  \exp\left(\kappa\sum_{\substack{p>K\\p\mid h}}\frac{1}{p}\right)
  \le e^\kappa.
\]
It remains only to prove that the small primes part has bounded average. Define
\[
  Z(a,b)=
  \prod_{\substack{p\le K^{1/2}\\p\mid a^a-b^b}}
  \left(1+\frac{\kappa}{p}\right).
\]
Then we have 
\[
  Z(a,b)=
  \sum_{\substack{d\ge 1\ \mathrm{squarefree}\\p\mid d\Rightarrow p\le K^{1/2}}}
  \frac{\kappa^{\omega(d)}}{d}\,\one_{d\mid a^a-b^b}.
\]
For the average over the off-diagonal pairs $(a,b)$ and any function $F(a, b)$, write
\[
  \mathbb E_K (F(a,b))
  =\frac{1}{K(K-1)}
  \sum_{\substack{1\le a,b\le K\\a\ne b}}F(a,b).
\]
The trivial divisor $d=1$ contributes $1$ in $\mathbb E_K (Z(a,b))$. For $d>1$, let $q=P^+(d)$ be the largest prime factor of $d$. Then $\one_{d\mid a^a-b^b} \le \one_{q\mid a^a-b^b}$. 
Although $B_q(K) =\#\{1\le a,b\le K:a^a\equiv b^b\pmod q\}$ includes the diagonal pairs $a=b$, using its upper bound estimate only enlarges the off-diagonal count.  By Lemma~\ref{lem:one-prime-collision}, for any prime $q\le K^{1/2}$, we have 
\[
  \mathbb E_K (\one_{q\mid a^a-b^b})
  \le \frac{B_q(K)}{K(K-1)}
  \ll \frac{\tau(q-1)}{q}.
\]
Consequently
\begin{align*}
  \mathbb E_K (Z(a, b))
  &\le 1+
  \sum_{\substack{q\le K^{1/2}\\q\ \mathrm{prime}}}
  \frac{\tau(q-1)}{q}
  \sum_{\substack{d\ \mathrm{squarefree}\\P^+(d)=q}}
  \frac{\kappa^{\omega(d)}}{d} \\
  &=1+
  \sum_{\substack{q\le K^{1/2}\\q\ \mathrm{prime}}}
  \frac{\tau(q-1)}{q}\cdot
  \frac{\kappa}{q}
  \prod_{p<q}\left(1+\frac{\kappa}{p}\right).
\end{align*}
By Mertens' estimate $\displaystyle \prod_{p<q}\left(1+\frac{\kappa}{p}\right) \ll_\kappa (\log q)^\kappa$, the divisor function bound $\tau(n)\ll_\eps n^\eps$ with any small $\eps > 0$, we obtain
\[
\mathbb E_K (Z(a, b))
\ll_\kappa 
1 + 
\sum_{q \text{ prime}}
\frac{\kappa \tau(q-1)(\log q)^\kappa}{q^2}
\ll_\kappa
\sum_{q \text{ prime}}
\frac{(\log q)^\kappa}{q^{2 - \eps}}
\ll_\kappa 
1.
\]
Together with the pointwise bounds for the middle and large prime ranges, this gives
\[
  \mathbb E_K \left(
  \prod_{p\mid a^a-b^b}\left(1+\frac{\kappa}{p}\right)
  \right)
  \ll_\kappa 1.
\]
This proves \eqref{eq:avg-singular}.
\end{proof}

\section{Proof of the main theorem}
\begin{proof}
Let $K=K(N)=\max\{a\in\N:a^a\le N/2\}$. Then $K\sim \frac{\log N}{\log\log N}$. Define the representation function
\begin{equation}\label{eq:RNdef}
  R_N(n)=\sum_{1\le a\le K}\one_{\Ptwo}(n-a^a),
\end{equation}
where the summand is interpreted as \(0\) when \(n-a^a<1\). We will prove
$\sum_{n\le N}R_N(n)\gg N$ and $\sum_{n\le N}R_N(n)^2\ll N$. 

For the first moment, using \eqref{eq:RNdef} and the inequality \(a^a\le N/2\), we get
\[
  \sum_{n\le N}R_N(n)
  =\sum_{1\le a\le K}\#\{m\le N-a^a:m\in\Ptwo\}.
\]
Since $N/2\le N-a^a\le N$ for $a\le K$, Landau's asymptotic
\eqref{eq:P2-asymp} is uniform in this range and gives
\[
  \#\{m\le N-a^a:m\in\Ptwo\}
  \sim (N-a^a)\frac{\log\log N}{\log N}.
\]
Moreover, we have $\displaystyle \sum_{1\le a\le K}a^a=O(N)$, 
because the last term is at most $N/2$ and the preceding terms form a rapidly decreasing tail. Hence
\begin{equation}\label{eq:first-moment}
  \sum_{n\le N}R_N(n)
  \sim
  \left(KN-\sum_{1\le a\le K}a^a\right)\frac{\log\log N}{\log N}
  \sim N.
\end{equation}

For the second moment, expanding the square gives
\begin{equation}\label{eq:second-expand}
  \sum_{n\le N}R_N(n)^2
  =
  \sum_{1\le a,b\le K}
  \#\{n\le N:n-a^a\in\Ptwo,\ n-b^b\in\Ptwo\}.
\end{equation}
The diagonal contribution \(a=b\) is bounded, by Landau's estimate, as
\begin{equation}\label{eq:diagonal}
  \sum_{1\le a\le K}\#\{n\le N:n-a^a\in\Ptwo\}
  \ll
  K\frac{N\log\log N}{\log N}
  \ll N.
\end{equation}
It remains to estimate the off-diagonal terms. Put \(h=a^a-b^b\ne0\). With \(m=n-a^a\), the condition \(n-b^b\in\Ptwo\) becomes \(m+h\in\Ptwo\). Since \(a^a,b^b\le N/2\), we have \(0<|h|\le N/2\). Lemma~\ref{lem:twodim} gives
\[
  \#\{m\le N:m+h\ge 1,\ m,m+h\in\Ptwo\}
  \ll
  \frac{N(\log\log N)^2}{(\log N)^2}\W_{\kappa_0}(h).
\]
Using Proposition~\ref{prop:avg-singular}, we obtain
\begin{align}
  &\sum_{\substack{1\le a,b\le K\\a\ne b}}
  \#\{n\le N:n-a^a\in\Ptwo,\ n-b^b\in\Ptwo\} \nonumber \\
  &\ll
  \frac{N(\log\log N)^2}{(\log N)^2}
  \sum_{\substack{1\le a,b\le K\\a\ne b}}
  \W_{\kappa_0}(a^a-b^b) \nonumber \\
  &\ll
  \frac{N(\log\log N)^2}{(\log N)^2}K^2 \nonumber \\
  &\ll N. \label{eq:offdiag3}
\end{align}
Combining \eqref{eq:second-expand}, \eqref{eq:diagonal}, and \eqref{eq:offdiag3}, we have
\begin{equation}\label{eq:second-moment}
  \sum_{n\le N}R_N(n)^2\ll N.
\end{equation}
Let $E_N=\{n\le N:R_N(n)>0\}$. 
By Cauchy's inequality, we have 
\[
  \left(\sum_{n\le N}R_N(n)\right)^2
  \le \#E_N\sum_{n\le N}R_N(n)^2.
\]
Together with \eqref{eq:first-moment} and \eqref{eq:second-moment}, this yields \(\#E_N\gg N\). Therefore there exists \(\delta>0\) such that, for all sufficiently large \(N\),
\[
  \#\{n\le N:n=m+a^a,\ m\in\Ptwo,\ a\ge1\}\ge \delta N.
\]
This proves Theorem~\ref{thm:main}. 
\end{proof}

\section*{Acknowledgments}
Huixi Li's research is supported by the National Natural Science Foundation of China (Grant No. 12561001). 
The authors thank Liangxun Li for helpful discussion.
ChatGPT was used as an auxiliary tool. All mathematical arguments, proofs, and computations were independently verified by the authors.

\end{document}